\documentclass[10pt,twoside]{siamart1116}

\usepackage[english]{babel}
\usepackage{graphicx,epstopdf,epsfig}
\usepackage{amsfonts,epsfig,fancyhdr,graphics,hyperref,amsmath,amssymb}
\usepackage{mathrsfs}

\newcommand{\HE}{Namie of Handling Editor}
\newcommand{\DoS}{Month/Day/Year}
\newcommand{\DoA}{Month/Day/Year}
\newcommand{\CA}{Name of Corresponding Author}

\newcommand{\Names}{Dockter et al.}
\newcommand{\Title}{Kronecker products of Perron similarities}

\newsiamremark{remark}{Remark}
\newsiamremark{example}{Example}

\DeclareMathOperator{\conv}{conv}
\DeclareMathOperator{\coni}{coni}
\DeclareMathOperator{\diag}{diag}

\allowdisplaybreaks

\begin{document}
\setcounter{page}{1}

\thispagestyle{empty}

\title{\Title\thanks{Received
by the editors on \DoS.
Accepted for publication on \DoA. 
Handling Editor: \HE. Corresponding Author: \CA}}

\author
{Janelle M.~Dockter\thanks{University of Washington Bothell, Bothell, WA
98011-8246, USA (jdockter@uw.edu, rlp5684@uw.edu, jdta@uw.edu).}
\and
Pietro Paparella\thanks{Division of Engineering and Mathematics, 
University of Washington Bothell, Bothell, WA
98011-8246, USA (pietrop@uw.edu).}
\and
Robert L.~Perry\footnotemark[2]
\and
Jonathan D Ta\footnotemark[2]
}

\pagestyle{myheadings}
\markboth{\Names}{\Title}

\maketitle

\begin{abstract}
An invertible matrix is called a {Perron similarity} if one of its columns and the corresponding row of its inverse are both nonnegative or both nonpositive. Such matrices are of relevance and import in the study of the {nonnegative inverse eigenvalue problem}. In this work, Kronecker products of Perron similarities are examined and used to to construct {ideal} Perron similarities all of whose rows are {extremal}. 
\end{abstract}



\section{Introduction}

An invertible matrix is called a \emph{Perron similarity} if one of its columns and the corresponding row of its inverse are both nonnegative or both nonpositive. Real Perron similarities were introduced by Johnson and Paparella \cite{jp2016,jp2017} and the case for complex matrices is forthcoming \cite{jp_psniep}. 

These matrices were introduced to examine the celebrated \emph{nonnegative inverse eigenvalue problem} vis-\'{a}-vis the polyhedral cone 
\[ \mathcal{C}(S) := \{ x \in \mathbb{R}^n \mid SD_x S^{-1} \ge 0 \} \] 
called the \emph{(Perron) spectracone of $S$}, and the set
\[ \mathcal{P}(S) := \left\{ x \in \mathcal{C}(S) \left\vert \begin{Vmatrix} x \end{Vmatrix}_\infty = 1 \right.\right\}, \]
called the \emph{(Perron) spectratope of $S$}. The latter is not necessarily a polytope, but in some cases is finitely-generated (this is true for some complex matrices as well). Notice that the entries of of any element in $\mathcal{P}(S)$ form a \emph{normalized spectrum} (i.e., $x_k = 1$ for some $k$ and $\max_i \{ \vert x_i \vert \} \le 1$) of a nonnegative matrix.

In particular, Johnson and Paparella \cite{jp2016} showed that if 
\[ H_n :=
\left\{
\begin{array}
{cl}
\begin{bmatrix} 
1 & 1 \\ 
1 & -1 
\end{bmatrix},     &  n = 2 \\
H_2 \otimes H_{n-1} =
\begin{bmatrix} 
H_{n-1} & H_{n-1} \\ 
H_{n-1} & -H_{n-1}
\end{bmatrix},     & n > 2 
\end{array} 
\right.,
\]
then $\mathcal{C}(H_n)$ and $\mathcal{P}(H_n)$ coincide with the conical hull and the convex hull of the rows of $H_n$, respectively. 

In this work, Kronecker products of Perron similarities are examined. In particular, it is shown that the Kronecker product of Perron similarities is a Perron similarity. An example is constructed to refute a result presented by Johnson and Paparella \cite[Corollary 3.17]{jp2016} (see Example \ref{ex:JPfail}). It is also shown that $\mathcal{C}(S)\otimes\mathcal{C}(T) \subset \mathcal{C}(S \otimes T)$ and $\mathcal{P}(S)\otimes\mathcal{P}(T) \subseteq \mathcal{P}(S \otimes T)$ (strict containment in the latter occurs for some matrices). Kronecker products of \emph{ideal} Perron similarities (see Section \ref{sec:ideal} below) yield Perron similarities all of whose rows are \emph{extremal}.

\section{Notation and Background}

For $a \in \mathbb{Z}$ and $n \in \mathbb{N}$, $a\bmod{n}$ is abbreviated to $a\%n$. For $n \in \mathbb{N}$, the set $\{ 1,\ldots,n\}$ is denoted by $\langle n \rangle$. 

The set of $m$-by-$n$ matrices over a field $\mathbb{F}$ is denoted by $\mathsf{M}_{m \times n}(\mathbb{F})$; when $m = n$, the set $\mathsf{M}_{m \times n}(\mathbb{F})$ is abbreviated to $\mathsf{M}_{n}(\mathbb{F})$. If $A \in \mathsf{M}_{m \times n}(\mathbb{F})$, then the $(i,j)$-entry of $A$ is denoted by $[A]_{ij}$, $a_{ij}$, or $a_{i,j}$.

In this work, $\mathbb{F}$ stands for $\mathbb{C}$ or $\mathbb{R}$. The set of $m$-by-$n$ matrices with entries over $\mathbb{F}$ is denoted by $\mathsf{M}_{m \times n}(\mathbb{F}) = \mathsf{M}_{m \times n}$; when $m = n$, $\mathsf{M}_{n \times n}(\mathbb{F})$ is abbreviated to $\mathsf{M}_n (\mathbb{F}) = \mathsf{M}_n$. The set of all $n$-by-$1$ column vectors is identified with the set of all ordered $n$-tuples with entries in $\mathbb{F}$ and thus denoted by $\mathbb{F}^n$. The set of nonsingular matrices in $\mathsf{M}_{n}$ is denoted by $\mathsf{GL}_n(\mathbb{F}) = \mathsf{GL}_n$.

Given $x \in \mathbb{F}^n$, $[x]_i = x_i$ denotes the $i$\textsuperscript{th} entry of $x$ and $\diag(x) = D_x = D_{x^\top} \in \mathsf{M}_{n}(\mathbb{F})$ denotes the diagonal matrix whose $(i,i)$-entry is $x_i$. Notice that for scalars $\alpha$, $\beta \in \mathbb{F}$, and vectors $x$, $y \in \mathbb{F}^n$, $D_{\alpha x + \beta y} = \alpha D_x + \beta D_y$. 

Denote by $I$, $e$, and $e_i$ the identity matrix, the all-ones vector, and the $i$\textsuperscript{th} canonical basis vector, respectively. The size of these objects are determined from the context in which they appear.

If $A \in \mathsf{M}_{m \times n}$ and $B \in \mathsf{M}_{p \times q}$, then the the \emph{Kronecker product of $A$ and $B$}, denoted by $A \otimes B$, is the $mp$-by-$nq$ matrix defined blockwise by $A \otimes B = \begin{bmatrix} a_{ij}B \end{bmatrix}$. More precisely, but less intuitively,  
\begin{equation}
    \label{eq:ijentry}
    [A \otimes B]_{ij} = a_{\lceil i/p \rceil, \lceil j/q \rceil}b_{[(i-1) \% p]+1,[(j-1) \% q]+1}.
\end{equation} 
If $x \in \mathbb{C}^m$ and $y \in \mathbb{C}^n$, then \eqref{eq:ijentry} simplifies to 
\begin{equation}
\label{eq: ientry vector}
    [x \otimes y]_{i} = x_{\lceil i/n \rceil}y_{(i-1) \%n + 1}.
\end{equation}
If $S, T \subseteq \mathbb{F}^n$, then $S \otimes T := \{ s \otimes t \mid s \in S,~t \in T \}$.

If $S\in \mathsf{GL}_n$, then the \emph{(Perron) spectracone of $S$}, denoted by $\mathcal{C}(S)$, is defined by $\mathcal{C}(S)=\{x \in \mathbb{F}^n \mid SD_xS^{-1} \ge 0\}$. The \emph{(Perron) spectratope of $S$} is the set \( \mathcal{P}(S) := \left\{ x \in \mathcal{C}(S) \left\vert \begin{Vmatrix} x \end{Vmatrix}_\infty = 1 \right.\right\}\). The conical hull and convex hull of the rows of $S$ are denoted by $\mathcal{C}_r(S)$ and $\mathcal{P}_r(S)$, respectively. If $S \in \mathsf{GL}_n$, then $S D_e S^{-1} = S I S^{-1} = I \ge 0$, i.e., $\emptyset \subset \coni(e) \subseteq \mathcal{C}(S)$.

If there is an $i \in \langle n \rangle$ such that $Se_i$ and $e_i^\top S^{-1}$ are both nonnegative or both nonpositive for $S \in \mathsf{GL}_n$, then $S$ is called a \emph{Perron similarity}.

\section{Preliminary Results}

\begin{lemma}
    \label{lem:altdivalg}
If $i \in \mathbb{Z}$ and $n \in \mathbb{N}$, then 
\[ 
i = (\lceil i/n \rceil - 1)n + (i-1)\%n + 1.
\]
\end{lemma}

\begin{proof}
By the division algorithm, 
\[ i-1 = \left\lfloor \frac{i-1}{n} \right\rfloor n + (i-1)\%n \]
and because
\[ \left\lfloor \frac{i-1}{n} \right\rfloor = \left\lceil \frac{i}{n} \right\rceil - 1, \]
it follows that 
\[ i-1 = (\lceil i/n \rceil - 1)n + (i-1)\%n, \]
i.e., 
\[ 
i = (\lceil i/n \rceil - 1)n + (i-1)\%n + 1. 
\]
\end{proof}

\begin{lemma} 
    \label{lem:kronprodbasis}
If $e_k\in \mathbb{F}^m$ and $e_\ell \in \mathbb{F}^n$, then $e_k\otimes e_\ell=e_{(k-1)n+\ell}\in \mathbb{F}^{mn}$.
\end{lemma}

\begin{proof}
It suffices to show that $[e_k\otimes e_\ell]_i = 1$ if and only if $i = (k-1)n+\ell$; to this end, if $i=(k-1)n+\ell$, then 
\[
\left\lceil \frac{i}{n} \right\rceil 
= \left\lceil \frac{(k-1)n+\ell}{n} \right\rceil 
= \left\lceil k - 1 + \frac{\ell}{n} \right\rceil 
= k
\]
and 
\begin{align*}
(i-1)\%n + 1 
&= ((k-1)n+\ell-1)\%n+1    \\
&= (\ell-1)\%n + 1          \\
&= \ell - 1 + 1 = \ell. 
\end{align*} 
Thus, according to \eqref{eq: ientry vector},
\begin{align*}
[e_k\otimes e_\ell]_i
=[e_k]_{\lceil \frac{i}{n}\rceil}[e_\ell]_{(i-1)\%n + 1}                                              
=[e_k]_{k}[e_\ell]_{\ell}                                                                              
=1.
\end{align*}

Conversely, if  
\[ 1 = [e_k\otimes e_\ell]_i =[e_k]_{\lceil \frac{i}{n}\rceil}[e_\ell]_{(i-1)\%n + 1}, \]
then $k = \lceil {i}/{n} \rceil$ and $\ell = (i-1)\%n + 1$. Hence, by the division algorithm, there is a positive integer $q$ such that $(i-1) = qn + \ell - 1$, i.e., $i = qn + \ell$.
Thus,
\begin{align*}
k
= \left\lceil \frac{qn + \ell}{n} \right\rceil
= \left\lceil q + \frac{\ell}{n} \right\rceil
= q + 1 
\end{align*}
i.e., $q = k - 1$. Therefore, $i = qn + \ell = (k-1)n + \ell$.
\end{proof}

\begin{lemma} 
    \label{lem:basisdecomp}
If $e_i \in \mathbb{F}^{mn}$, then $e_i = e_{\lceil i/n \rceil} \otimes e_{(i-1)\%n + 1}$, where $e_{\lceil i/n \rceil} \in \mathbb{F}^m$ and $e_{(i-1)\%n + 1} \in \mathbb{F}^n$.
\end{lemma}

\begin{proof}
If $e_{\lceil i/n \rceil} \in \mathbb{F}^m$ and $e_{(i-1)\%n + 1} \in \mathbb{F}^n$, then 
\[ 
e_{\lceil i/n \rceil} \otimes e_{(i-1)\%n + 1} = e_{(\lceil i/n \rceil - 1)n + (i-1)\%n + 1} = e_i
\]
by Lemmas \ref{lem:altdivalg} and \ref{lem:kronprodbasis}.
\end{proof}

\begin{lemma}
If $S \in \mathsf{M}_{m \times n}$ and $T \in \mathsf{M}_{p \times q}$, then
\begin{align*}
e_i^\top(S\otimes T) = e_{\lceil \frac{i}{p} \rceil}^\top (S) \otimes e_{(i-1)\% p+1}^\top (T).
\end{align*}
\end{lemma}

\begin{proof}
By Lemma \ref{lem:basisdecomp} and properties of the Kronecker product,
\begin{align*}
    e_i^\top (S \otimes T)
    &= (e_{\lceil \frac{i}{p} \rceil} \otimes e_{(i-1)\% p+1})^\top (S \otimes T)       \\
    &= (e_{\lceil \frac{i}{p} \rceil}^\top \otimes e_{(i-1)\% p+1}^\top) (S \otimes T)  \\
    &= e_{\lceil \frac{i}{p} \rceil}^\top (S) \otimes e_{(i-1)\%p+1}^\top (T).            
\end{align*}
\end{proof}

\begin{lemma} 
    \label{Lem: kron of diagonals}
If $x \in \mathbb{F}^m$ and $y \in \mathbb{F}^n$, then $D_x \otimes D_y = D_{x \otimes y}$.
\end{lemma}

\begin{proof}
If $i,j \in \langle mn \rangle$, then
\begin{align*}
[D_x \otimes D_y]_{ij} 
= [D_x]_{\big \lceil \frac{i}{n} \big \rceil, \big \lceil \frac{j}{n} \big \rceil} [D_y]_{(i-1)\%n+1,(j-1)\%n+1}
\end{align*}
in view of \eqref{eq:ijentry}. 
Since 
\[ 
[D_x]_{ij} =
\begin{cases}
x_i,& i = j \\
0,  & i \ne j,
\end{cases}
\]
it follows that $[D_x \otimes D_y]_{ij} \ne 0$ if and only if $\lceil {i}/{n} \rceil = \lceil {j}/{n} \rceil$ and $(i-1)\%n+1=(j-1)\%n+1$. These equations hold, in light of Lemma \ref{lem:altdivalg}, if and only if $i=j$. Thus, $D_x \otimes D_y$ is a diagonal matrix and, when $i=j$, notice that
\begin{align*}
[D_x]_{\big \lceil \frac{i}{n} \big \rceil, \big \lceil \frac{i}{n} \big \rceil} [D_y]_{(i-1)\%n+1,(i-1)\%n+1} 
&= x_{\lceil \frac{i}{n} \rceil} y_{(i-1)\%n+1}                                                                                                 \\
&= [x \otimes y]_i                                                                                                                              \\
&= [D_{x \otimes y}]_{ii},
\end{align*}
as required. 
\end{proof}

\begin{lemma}
    \label{lem:kronnorms}
If $x \in \mathbb{F}^m$ and $y \in \mathbb{F}^n$, then $\begin{Vmatrix}x\otimes y\end{Vmatrix}_p = \begin{Vmatrix}x\end{Vmatrix}_p \begin{Vmatrix}y\end{Vmatrix}_p$, $\forall p \in [1,\infty]$.
\end{lemma}

\begin{proof}
Notice that
\begin{align*}
\begin{Vmatrix} x\otimes y \end{Vmatrix}_p 
&= \sqrt[p]{\sum_{k=1}^{mn} \left\vert [x\otimes y]_k \right\vert^p}                                        \\
&= \sqrt[p]{\sum_{k=1}^{mn} \lvert x_{\lceil \frac{k}{n} \rceil} y_{(k-1)\%n + 1} \rvert^p}                 \\
&= \sqrt[p]{\sum_{k=1}^{mn} \lvert x_{\lceil \frac{k}{n} \rceil} \rvert^p \lvert y_{(k-1)\%n+1} \rvert^p}. 
\end{align*}
Since $\lceil {k}/{n} \rceil \in \langle m \rangle$ and $(k-1)\%n+1 \in \langle n \rangle$, it follows that
\begin{align*}
\begin{Vmatrix}x\otimes y \end{Vmatrix}_p
&= \sqrt[p]{\sum_{i=1}^m \sum_{j=1}^n \lvert x_i \rvert^p \lvert y_j \rvert^p}                              \\
&= \sqrt[p]{\sum_{i=1}^m \left(\lvert x_i \rvert^p \left(\sum_{j=1}^n \lvert y_j \rvert^p \right)\right)}   \\
&= \sqrt[p]{\left(\sum_{i=1}^m \lvert x_i \rvert^p \right) \left(\sum_{j=1}^n \lvert y_j \rvert^p \right)}  \\
&= \sqrt[p]{\sum_{i=1}^m \lvert x_i \rvert^p} \cdot \sqrt[p]{\sum_{j=1}^n \lvert y_j \rvert^p}              
= \begin{Vmatrix}x\end{Vmatrix}_p \begin{Vmatrix}y\end{Vmatrix}_p
\end{align*}
The case when $p = \infty$ follows from the fact that $\begin{Vmatrix}x\end{Vmatrix}_\infty = \lim_{p\to\infty}\begin{Vmatrix}x\end{Vmatrix}_p$.
\end{proof}

\section{Main Results}

\begin{theorem}
    \label{thm:subsetcontain}
If $S\in \mathsf{GL}_m$ and $T\in \mathsf{GL}_n$, then $\mathcal{C}(S) \otimes \mathcal{C}(T) \subseteq \mathcal{C}(S \otimes T)$ and $\mathcal{P}(S) \otimes \mathcal{P}(T) \subseteq \mathcal{P}(S \otimes T)$.
\end{theorem}

\begin{proof}
If $z \in \mathcal{C}(S) \otimes \mathcal{C}(T)$, then $z = x \otimes y$, where $x \in \mathcal{C}(S)$ and $y \in \mathcal{C}(T)$. Thus, \(SD_xS^{-1} \ge 0\) and \(TD_y T^{-1} \ge 0\)
By Lemma \ref{Lem: kron of diagonals} and properties of the Kronecker product,
\begin{align*}
(SD_xS^{-1})\otimes (TD_y T^{-1}) 
&= (S \otimes T)(D_x \otimes D_y)(S^{-1} \otimes T^{-1})    \\
&= (S \otimes T)(D_{x \otimes y})(S \otimes T)^{-1}\geq 0,
\end{align*}
since the Kronecker product of nonnegative vectors is nonnegative.
Therefore, $z \in C(S \otimes T)$ and $\mathcal{C}(S) \otimes \mathcal{C}(T) \subseteq C(S\otimes T)$. 

If, in addition, $z \in \mathcal{P}(S) \otimes \mathcal{P}(T)$, then $x \in \mathcal{P}(S)$ and $y \in \mathcal{P}(T)$, i.e., $\begin{Vmatrix}x\end{Vmatrix}_\infty = \begin{Vmatrix}y\end{Vmatrix}_\infty = 1$. By Lemma \ref{lem:kronnorms}, 
\begin{align*}
\begin{Vmatrix}z\end{Vmatrix}_\infty 
= \begin{Vmatrix}x \otimes y\end{Vmatrix}_\infty 
= \begin{Vmatrix}x\end{Vmatrix}_\infty \begin{Vmatrix}y\end{Vmatrix}_\infty 
= 1
\end{align*}
i.e., $z \in P(S\otimes T)$ and $P(S) \otimes P(T) \subseteq P(S\otimes T)$.
\end{proof}

\begin{theorem}
\label{thm: Perron Similarity}
If $S\in \mathsf{GL}_m$ and $T\in \mathsf{GL}_n$ are Perron similarities, then $S\otimes T$ is a Perron similarity.
\end{theorem}

\begin{proof}
By definition, $\exists k\in \left< m\right>$ and $\exists \ell \in \left< n\right>$ such that the vectors $Se_k$, $e_k^\top S^{-1}$, $Te_\ell$, and $e_\ell^\top T^{-1}$ are nonnegative. By Lemma \ref{lem:kronprodbasis},  
\begin{align*}
(S\otimes T) e_{(k-1)n+\ell}
= (S\otimes T)(e_k\otimes e_\ell) 
&= Se_k\otimes Te_\ell\geq 0
\end{align*}
and
\begin{align*}
e_{(k-1)n+\ell}^\top (S\otimes T)^{-1} 
&= (e_k\otimes e_\ell)^\top(S\otimes T)^{-1}            \\
&=(e_k^\top\otimes e_\ell^\top)(S^{-1}\otimes T^{-1})   \\
&=e_k^\top S^{-1}\otimes e_\ell^\top T^{-1} \geq 0.
\end{align*}
Therefore, $S\otimes T$ is a Perron similarity.
\end{proof}





\begin{remark} 
    \label{rem:convexcone}
If $x,y \in \mathcal{C}(S)$ and $\alpha,\beta \geq 0$, then $\alpha x + \beta y \in \mathcal{C}(S)$, i.e., $\mathcal{C}(S)$ is a convex cone.
\end{remark}

\begin{remark} 
    \label{rem: PS nontrivial spec}
If $S\in \mathsf{GL}_n$ is a Perron similarity, then 
\[ S D_{e_i} S^{-1} = (Se_i)(e_i^\top S^{-1}) \ge 0, \] 
i.e., $\exists x\in \mathcal{C}(S)$ such that $x\neq \alpha e$ for every nonnegative $\alpha$.
\end{remark}

\begin{lemma}
    \label{lem:nz}
If $S$ is a Perron similarity, then there is a vector $x \in \mathcal{C}(S)$ such that $x$ is totally nonzero and not a scalar multiple of $e$.
\end{lemma}

\begin{proof}
By Remark \ref{rem: PS nontrivial spec}, $\exists x'\in \mathcal{C}(S)$ such that $x'$ is not a scalar multiple of $e$. Select $\alpha\in \mathbb{R}$ such that $\alpha >\max_i |\Re(x'_i)|$. If $x := x'+\alpha e$, then $x$ is totally nonzero, not a scalar multiple of $e$, and belongs to $\mathcal{C}(S)$ by Remark \ref{rem:convexcone}.
\end{proof}

\begin{theorem}
    \label{thm:conecontain}
If $S\in \mathsf{GL}_m$ and $T\in \mathsf{GL}_n$ are Perron similarities such that $m>1$ and $n>1$, then $\mathcal{C}(S)\otimes \mathcal{C}(T)\subset \mathcal{C}(S\otimes T)$. 
\end{theorem}

\begin{proof}
By Lemma \ref{lem:nz}, we may select $x\in \mathcal{C}(S)$ and $y\in \mathcal{C}(T)$ such that $x$ and $y$ are totally nonzero and not scalar multiples of $e$. If $z := x \otimes y$, then $z\in \mathcal{C}(S\otimes T)$ by Theorem \ref{thm:subsetcontain}. As $x$ and $y$ are not scalar multiples of $e$, there are integers $i,j \in \langle m\rangle$ and $k,\ell \in \langle n\rangle$ such that $x_i\neq x_j$ and $y_k\neq y_\ell$. Notice that the vector $z$ contains the blocks $x_i y$ and $x_j y$ so the vector $z$ has the entries $z_\alpha = x_iy_k$, $z_\beta=x_iy_\ell$, $z_\gamma=x_j y_k$, and $z_\delta = x_j y_\ell$. Furthermore, notice that   \[\frac{z_\alpha}{z_\beta}=\frac{y_k}{y_\ell}=\frac{z_\gamma}{z_\delta}. \] 

Select $\varepsilon > 0$ such that $z' := z + \varepsilon e$ is totally nonzero. By Remark \ref{rem:convexcone}, $z' \in \mathcal{C}(S \otimes T)$. For contradiction, assume that $z' = x' \otimes y'$, where $x'\in \mathcal{C}(S)$ and $y'\in \mathcal{C}(T)$. The vectors $x'$ and $y'$ must be totally nonnzero (otherwise, $z'$ would not be totally nonzero). By a similar argument above,  
\[ \frac{z'_\alpha}{z'_\beta}=\frac{y'_k}{y'_\ell}=\frac{z'_\gamma}{z'_\delta}. \]
However,
\begin{align*}
\frac{z'_\alpha}{z'_\beta} &= \frac{z'_\gamma}{z'_\delta}                                                                                           \\
\iff 
\frac{z_\alpha + \varepsilon}{z_\beta + \varepsilon} &= \frac{z_\gamma+\varepsilon}{z_\delta+\varepsilon}                                           \\
\iff 
\frac{x_iy_k+\varepsilon}{x_iy_\ell+\varepsilon}&=\frac{x_jy_k+\varepsilon}{x_jy_\ell+\varepsilon}                                                  \\
\iff 
(x_iy_k+\varepsilon)(x_jy_\ell+\varepsilon) &= (x_iy_\ell+\varepsilon)(x_jy_k+\varepsilon)                                                          \\
\iff 
x_ix_jy_ky_\ell+\varepsilon x_iy_k+\varepsilon x_jy_\ell+\varepsilon ^2&=x_ix_jy_ky_\ell+\varepsilon x_iy_\ell+\varepsilon x_jy_k+\varepsilon^2     \\
\iff 
\varepsilon (x_i-x_j)(y_k-y_\ell) &=0                                                                                                               \\
\iff 
x_i-x_j = 0~\text{or}~y_k-y_\ell &= 0,
\end{align*}
a contradiction. Thus, $\mathcal{C}(S)\otimes \mathcal{C}(T)\subset C(S\otimes T)$.  
\end{proof}

\begin{example}
    \label{ex:JPfail}
Johnson and Paparella \cite[Corollary 3.17]{jp2016} stated that $S$ is a Perron similarity if and only if $\coni(e)$ is properly contained in $\mathcal{C}(S)$. A contribution of this work is the refutation of this result with a counterexample constructed via the Kronecker product. 

Indeed, the matrix
\begin{equation*}
S := 
\begin{bmatrix}
1 & 2 & 1 & 2   \\
1 & 1 & 1 & 1   \\
1 & 2 & -1 & -2 \\
1 & 1 & -1 & -1    
\end{bmatrix}
\end{equation*}
is the Kronecker product of 
\[ 
H_2 =
\begin{bmatrix}
1 &  1 \\
1 & -1
\end{bmatrix}
\]
and 
\[ 
T :=
\begin{bmatrix}
1 &  2 \\
1 & 1
\end{bmatrix}
\]
The inverse of $S$ is 
\[
\begin{bmatrix}
-0.5 & 1 & -0.5 & 1     \\
0.5 & -0.5 & 0.5 & -0.5 \\
-0.5 & 1 & 0.5 & -1     \\
0.5 & -0.5 & -0.5 & 0.5
\end{bmatrix}
 \]
Notice that neither the first or second row is nonnegative. Furthermore, if $D = \diag(\begin{bmatrix} 2 & 2 & -1 & -1 \end{bmatrix})$, then the matrix
\[
A := SDS^{-1} =
\begin{bmatrix}
0.5 & 0 & 1.5 & 0   \\
0 &  0.5 & 0 & 1.5  \\
1.5 & 0 & 0.5 & 0   \\
0 &  1.5 & 0 & 0.5
\end{bmatrix} \]
is nonnegative and nonscalar. Thus, $\coni(e)$ is properly contained in $\mathcal{C}(S)$, but $S$ is not a Perron similarity. 
\end{example}

\section{Ideal Perron Similarities}
    \label{sec:ideal}

If $S \in \mathsf{GL}_n$ is a Perron similarity, then $S$ is called \emph{ideal} if $\mathcal{C}(S) = \mathcal{C}_r(S)$. For real matrices, it is known that $S$ is ideal if and only if $\exists k \in \langle n \rangle$ such that $e_k^\top S = e^\top$ and $e_i^\top S \in \mathcal{C}(S)$ for all $i \in \langle n \rangle$ \cite[Theorem 3.8]{jp2017}. A careful examination of the arguments also applies to complex matrices. 

\begin{theorem} 
    \label{thm:ideal}
If $S\in \mathsf{GL}_m$ and $T\in \mathsf{GL}_n$ are ideal, then $S\otimes T$ is ideal.
\end{theorem}

\begin{proof}
By hypothesis, there are integers $k \in \langle m \rangle$ and $\ell \in \langle n \rangle$ such that $e_k^\top S = e$ and $e_\ell^\top T = e$. Notice that $(k-1)n + \ell \in \langle mn \rangle$ and by Lemma \ref{lem:kronprodbasis}, 
\begin{align*}
e_{(k-1)n + \ell}^\top (S \otimes T) 
&= (e_k \otimes e_\ell)^\top (S \otimes T)                                  \\
&= (e_k^\top \otimes e_\ell^\top) (S \otimes T)                             \\
&= (e_k^\top S) \otimes (e_\ell^\top T) = e^\top \otimes e^\top = e^\top. 
\end{align*}

If $i \in \langle mn \rangle$, then, following Lemma \ref{lem:basisdecomp},
\begin{align*}
e_i^\top (S \otimes T) 
&= (e_{\lceil i/n \rceil} \otimes e_{(i-1)\%n + 1})^\top (S \otimes T)      \\
&= (e_{\lceil i/n \rceil}^\top \otimes e_{(i-1)\%n + 1}^\top) (S \otimes T) \\
&= (e_{\lceil i/n \rceil}^\top S) \otimes (e_{(i-1)\%n + 1}^\top T)         
\ge 0
\end{align*}
since $e_{\lceil i/n \rceil}^\top S \in \mathcal{C}(S)$ and $e_{(i-1)\%n + 1}^\top T \in \mathcal{C}(T)$. 
\end{proof}

\begin{theorem}
    \label{thm:conicontain}
If $U = \{u_1,\ldots, u_p\} \subseteq \mathbb{F}^m$ and $V = \{v_1,\ldots, v_q\} \subseteq \mathbb{F}^n$, then $\coni(U) \otimes \coni(V) \subseteq \coni(U \otimes V)$ and $\conv (U) \otimes \conv (V) \subseteq \conv (U\otimes V)$.
\end{theorem}

\begin{proof}
If $x \in \coni(U) \otimes \coni(V)$, then $x = u\otimes v$, where $u \in \coni(U)$ and $v \in \coni(V)$. By definition,
\begin{align*}
    u = \sum_{i=1}^p \lambda_i u_i,~\lambda_i \ge 0,~\forall i \in \langle p \rangle
\end{align*}
and 
\begin{align*}
    v = \sum_{j=1}^q \mu_j v_j,~\mu_j \ge 0,~\forall j \in \langle q \rangle.
\end{align*}
By properties of the Kronecker product,
\begin{align*}
x = u\otimes v 
&= \left( \sum_{i=1}^p \lambda_i u_i \right) \otimes \left( \sum_{j=1}^q \mu_j v_j \right)  \\
&= \sum_{i=1}^p \left( \lambda_i u_i \otimes \sum_{j=1}^q (\mu_j v_j) \right)               \\
&= \sum_{i=1}^p \sum_{j=1}^q \lambda_i \mu_j (u_i \otimes v_j) \in \coni(U)\otimes \coni(V) 
\end{align*}
since $\lambda_i \mu_j \geq 0$, $\forall (i,j) \in \langle p \rangle \times \langle q \rangle$.

If, in addition, 
\begin{align*}
    \sum_{i=1}^p \lambda_i = \sum_{j=1}^q \mu_j = 1,
\end{align*}
then 
\begin{align*}
\sum_{i=1}^p \sum_{j=1}^q \lambda_i \mu_j 
= \sum_{i=1}^p \left( \lambda_i \sum_{j=1}^q \mu_j \right)
= \sum_{i=1}^p \lambda_i 
= 1,
\end{align*}
i.e., $\conv(U) \otimes \conv(V) \subseteq \conv(U\otimes V)$.
\end{proof}

Recall that a matrix is \emph{irreducible} if and only if its digraph is strongly connected (see, e.g., Brualdi and Ryser \cite[Theorem 3.2.1]{br1991}). The \emph{index of imprimitivity} of an irreducible matrix is the greatest common divisor of the lengths of the closed directed walks in its digraph \cite[p.~68]{br1991}.

An invertible matrix $S$ is called \emph{strong} if there is an {irreducible} nonnegative matrix $A$ such that $A = S D S^{-1}$ (in such a case, $S$ must be a Perron similarity since the eigenspace corresponding to the Perron root is one-dimensional). If $S$ is strong, then $S$ is ideal if and only if $\mathcal{P}(S) = \mathcal{P}_r(S)$ \cite{jp_psniep}.

The following result is a consequence of a result stated by Harary and Trauth \cite[p.~251]{ht1966} and follows from a result due to McAndrew \cite[Theorem 2]{m1963}. 

\begin{theorem}
    \label{thm:kronirr}
If $A$ and $B$ are irreducible and $k$ and $\ell$ are the indices of imprimitivity of $A$ and $B$, respectively, then $A \otimes B$ is irreducible if and only if $\gcd(k,\ell) = 1$.
\end{theorem}

\begin{corollary}
Suppose that $S$ and $T$ are ideal and strong. Let $A$ and $B$ be irreducible nonnegative matrices with relatively prime indices of imprimitivity $k$ and $\ell$, respectively, and such that $A = SDS^{-1}$ and $B = T\hat{D}T^{-1}$. Then $S \otimes T$ is ideal, strong, and $\mathcal{P}(S)\otimes\mathcal{P}(T) \subset \mathcal{P}(S \otimes T)$.
\end{corollary}   

\begin{proof}
The matrix $S \otimes T$ is ideal by Theorem \ref{thm:ideal} and strong by Theorem \ref{thm:kronirr}. Thus, $\mathcal{P}_r(S) = \mathcal{P}(S)$, $\mathcal{P}_r(T) = \mathcal{P}(T)$, and $\mathcal{P}_r(S \otimes T) = \mathcal{P}(S \otimes T)$. The weak containment $\mathcal{P}(S)\otimes\mathcal{P}(T) \subseteq \mathcal{P}(S \otimes T)$ follows from Theorem \ref{thm:conicontain}.

By Lemma \ref{lem:nz}, we may select $\hat{x}\in \mathcal{C}(S)$ and $\hat{y}\in \mathcal{C}(T)$ such that $\hat{x}$ and $\hat{y}$ are totally nonzero and not scalar multiples of $e$. Furthermore, the totally nonzero vectors $x := \hat{x}/||x||_\infty$ and $y := \hat{y}/||y||_\infty$ belong to $\mathcal{P}(S)$. If $z := x \otimes y$, then $z\in \mathcal{P}(S\otimes T)$ by Theorem \ref{thm:subsetcontain}. As $x$ and $y$ are not scalar multiples of $e$, there are integers $i,j \in \langle m\rangle$ and $k,\ell \in \langle n\rangle$ such that $x_i\neq x_j$ and $y_k\neq y_\ell$. Notice that the vector $z$ contains the blocks $x_i y$ and $x_j y$ so the vector $z$ has the entries $z_\alpha = x_iy_k$, $z_\beta=x_iy_\ell$, $z_\gamma=x_j y_k$, and $z_\delta = x_j y_\ell$. Furthermore, notice that   \[\frac{z_\alpha}{z_\beta}=\frac{y_k}{y_\ell}=\frac{z_\gamma}{z_\delta}. \] 

Since $\mathcal{P}(S \otimes T) = \mathcal{P}_r(S \otimes T)$, it follows that $\mathcal{P}(S \otimes T)$ is convex. As $x$ and $y$ are not multiples of $e$, they are not on the ray passing through $e$. Thus, we may select $\varphi, \psi > 0$ such that $z' := \varphi z + \psi e$ is totally nonzero, $\varphi + \psi = 1$, and $z' \in \mathcal{C}(S \otimes T)$. 

For contradiction, assume that $z' = x' \otimes y'$, where $x'\in \mathcal{P}(S)$ and $y'\in \mathcal{P}(T)$. By a similar argument above,  
\[ \frac{z'_\alpha}{z'_\beta}=\frac{y'_k}{y'_\ell}=\frac{z'_\gamma}{z'_\delta}. \] 
However,
\begin{align*}
\frac{z'_\alpha}{z'_\beta} &= \frac{z'_\gamma}{z'_\delta}                                                                                                   \\
\iff 
\frac{\varphi z_\alpha + \psi}{\varphi z_\beta + \psi} &= \frac{\varphi z_\gamma+\psi}{\varphi z_\delta+\psi}                                               \\
\iff 
\frac{\varphi x_iy_k+\psi}{\varphi x_iy_\ell+\psi}&=\frac{\varphi x_j y_k + \psi}{\varphi x_j y_\ell+\psi}                                                  \\
\iff 
(\varphi x_iy_k + \psi)(\varphi x_j y_\ell + \psi) &= (\varphi x_i y_\ell + \psi)(\varphi x_j y_k + \psi)                                                   \\
\iff 
\varphi^2 x_i x_j y_k y_\ell + \varphi\psi (x_i y_k + x_j y_\ell) + \psi ^2 &= \varphi^2 x_i x_j y_k y_\ell + \varphi\psi (x_i y_\ell + x_j y_k) + \psi^2   \\
\iff 
\varphi\psi (x_i-x_j)(y_k-y_\ell) &=0                                                                                                                              \\
\iff x_i-x_j = 0~\text{or}~y_k-y_\ell &= 0,
\end{align*}
a contradiction. Thus, $\mathcal{P}(S)\otimes \mathcal{P}(T)\subset \mathcal{P}(S\otimes T)$.  
\end{proof}

\begin{example}
For $n \in \mathbb{N}$, let $F = F_n$ be the \emph{discrete Fourier transform} matrix of order $n$, i.e., $F$ is the $n$-by-$n$ matrix with $(i,j)$-entry equal to $\omega^{(i-1)(j-1)}$, where $\omega := \exp(2 \pi i/n)$. Notice that
\[ F =
\begin{bmatrix}
1       & 1             & \cdots    & 1                 & \cdots    & 1                 \\
1       & \omega        & \cdots    & \omega^k          & \cdots    & \omega^{n-1}      \\
\vdots  & \vdots        & \ddots    & \vdots            & \vdots    & \vdots   		    \\
1       & \omega^k      & \cdots    & \omega^{k^2}      & \cdots    & \omega^{k(n-1)}   \\
\vdots  & \vdots        & \vdots    & \vdots            & \ddots    & \vdots            \\
1       & \omega^{n-1}  & \cdots    & \omega^{k(n-1)}   & \cdots    & \omega^{(n-1)^2} 
\end{bmatrix} \]
and $F$ is ideal as it is a Vandermonde matrix corresponding to the polynomial $p(t) := t^n - 1$. The companion matrix $C$ corresponding to $p$ is nonnegative and the spectrum of the nonnegative matrix $C^{k-1}$ corresponds to the $k$\textsuperscript{th}-row of $F$, $k \in \langle n \rangle$. Furthermore, $F$ is strong given that $C$ is the adjacency matrix of the directed cycle of length $n$ and, hence, is irreducible (it also admits positive circulant matrices).

A normalized, realizable spectrum $x$ is called \emph{extremal} if $\alpha x$ is not realizable whenever $\alpha > 1$. Notice that every row of $F$ is extremal and every point in every row is extremal in the Karpelevi{\v{c}} region. 

At the 2019 Meeting of the International Linear Algebra Society in Rio de Janeiro, the second author asked whether other such matrices exist. Notice that $F_n \otimes F_m$, $F_m \otimes H_n$, and $H_n \otimes F_m$ are matrices all of whose rows and entries are extremal. .  
\end{example}

\bibliographystyle{plain}
\bibliography{kp}

\end{document}